\documentclass[12 pt]{article}
\usepackage{amssymb,amsmath,amsthm,latexsym}
\usepackage{times}
\usepackage{color}
\usepackage{float}
\usepackage[caption = false]{subfig}
\usepackage{graphicx}
\usepackage{wrapfig}
\usepackage{epstopdf}
\usepackage{amssymb}
\usepackage{amsmath}
\usepackage{array}
\usepackage{caption}
\usepackage[margin=1cm]{caption}
\usepackage{authblk}
\usepackage[utf8]{inputenc}
\newtheorem{theorem}{Theorem}
\newtheorem{lemma}{Lemma:}

\renewcommand{\thefootnote}{\fnsymbol{footnote}}
\title{Qualitative Analysis and Optimal Control Strategy of an SIR Model with Saturated Incidence and Treatment}
\author[1]{Jayanta Kumar Ghosh }
\author[2]{Uttam  Ghosh \thanks{Corresponding author Email:uttam\_math@yahoo.co.in}}
\author[3]{M. H. A. Biswas}
\author[2]{Susmita Sarkar}
\affil[1]{Boalia Junior High School, Nadia, West Bengal, India.}
\affil[2]{Department of Applied Mathematics, University of Calcutta, Kolkata, India.}
\affil[3]{Mathematics Discipline, Khulna University, Khulna-9208, Bangladesh}
\begin{document}

\maketitle
\renewcommand{\thefootnote}{\arabic{footnote}}
\abstract{This paper deals with an SIR model with saturated incidence rate affected by inhibitory effect and saturated treatment function. Two control functions have been used, one for vaccinating the susceptible population and other for the treatment control of infected population. We have analysed the existence and stability of equilibrium points and investigated the transcritical and backward bifurcation. The stability analysis of non-hyperbolic equilibrium point has been performed by using Centre manifold theory. The Pontryagin's maximum principle has been used to characterize the optimal control whose numerical results show the positive impact of two controls mentioned above for controlling the disease. Efficiency analysis is also done to determine the best control strategy among vaccination and treatment.}


\section{Introduction}
Mathematical modelling in epidemiology have become powerful and important tool to understand the infectious disease dynamics and to improve control of infection in the population. A good epidemic model is an intelligent model which is able to predict any possible outbreak of the disease and is effective in reducing the transmission of the disease. It is a simplest version of reality in Biology [1-3].\\
In mathematical epidemiology, the incidence rate as well as treatment rate plays a crucial role while analysing the transmission of infectious diseases. The researchers in this field consider different type of incidence rate depending on character of disease spreading. Firstly, the bi-linear incidence rate [4] $ \beta SI $, (where the parameter $ \beta $ is transmission rate of infection and the variables $ S,I $ are respectively the number of susceptible and infected population) is based on the law of mass action which is not realistic because, at the initial stage of spreading of disease, the number of infected population is low and so people do not take care at this stage but at the later stage when the number of infected population becomes large, people take more care to protect them from the infection of the disease. The saturated incidence rate  $ \frac{\beta SI}{1+\alpha I} $ was introduced by Anderson and May in 1978 [5], where $ \alpha $ is defined as inhibitory coefficient. Clearly, this incidence rate is an increasing function of $ S $ as well as $ I $ and by this incidence rate the total growth of infected population is less compared to the standard incidence. This type of infection sometimes named as `incidence rate with psychological effect' [6], because the effect of $ \alpha $ stems from epidemic control (taking appropriate preventive measures and awareness) and the rate of infection decreases as the inhibitory coefficient $ \alpha $ increases.\\
Again, we are aware of the fact that the treatment is an important method to control diseases. The treatment rate of infected individuals is considered to be either constant or proportional to the number of infected individuals. Wang and Ruan [7] introduced a constant treatment function $ T(I) $ in an SIR model, where 
$ T(I) =\begin{cases} r,\ I>0\\0, \ \ I=0 \end{cases}$.
Again, we know that there are limited treatment resources or limited capacity for treatment in every community. To include this type of limitations in treatment Zhang and Liu [8] introduced a new continuously differentiable treatment function $ T(I)=\frac{rI}{1+\alpha I} $ to characterize the saturation phenomenon of the limited medical resources. Here $ T(I) $ is an increasing function of $ I $ and $ \frac{r}{\alpha} $ is the maximal supply of medical resource per unit time.\\
On the other hand, optimal control theory is a powerful mathematical tool that is used extensively to control the spread of infectious diseases. It is often used in the control of the spread of most infectious diseases for which either vaccine or treatment is available. Some researchers considered only vaccination control to their models [9] and some of them  used only treatment control [9-10]. The author in [11] have used both the controls in their models. The purpose of considering both vaccination and treatment in finding optimal control in epidemiological models is to minimize the susceptible and infected individuals as well as the cost of implementing these two controls.\\
In this paper, we have considered an SIR model with saturated incidence rate $ \frac{\beta SI}{1+\alpha I} $ affected by the inhibitory effect $ \alpha $ and saturated treatment function. Both vaccination control $ u_{1} $ and treatment control $ u_{2} $ have been used to address the question of how to optimally combine the vaccination and treatment strategies for minimizing the susceptible and infected individuals as well as the cost of the implementation of the two interventions. Here, we have used the treatment function $ T(I,u_{2})=\frac{ru_{2}I}{1+bu_{2}I} $, which is clearly the increasing function of $ I $ and $ u_{2} $ and the maximal supply of medical resource is $ \frac{r}{b} $, where $ b $ is the delayed parameter of treatment because $ T $ decreases as $ b $ increases and $ r $ is the cure rate. Here, we have analysed the stability of equilibrium points using eigen analysis method. The stability analysis of non-hyperbolic equilibrium point has been carried out by using Center manifold theory. Exhibition of transcritical and backward bifurcation have been analysed in our work. It is important to mention here that our work is different from some of the other related works cited in this paper[10-11] because the stability or instability of endemic equilibrium point(s) is analysed by applying different techniques described in [12]. It should also be noted that in this paper, we shall deal with the qualitative analysis of the model as well as the optimal control of the disease. Numerical simulations and efficiency analysis are performed to understand the positive impact of controls and to determine the best strategy among vaccination and treatment.\\ Organization of the paper is as follows. We have formulated the model in Section 2 and  discussed the boundedness of the solutions, existence of equilibria and basic reproduction number $ R_{0} $ in Section 3. The Section 4 is devoted to the stability and bifurcation analysis about disease free equilibrium point and Section 5 is devoted to the backward bifurcation and stability analysis of endemic equilibrium points. Section 6 gives detailed description about the characterization of optimal control. The numerical simulations and efficiency analysis are given in Section 7 and final section gives the conclusions.
\section{Model Formulation}
\paragraph{\textnormal{Let the total population be divided into three classes, namely susceptible population $ S(t) $, infected population $ I(t) $ and recovered population $ R(t) $ at time $ t $. Here, we have considered an epidemic model in which the birth rate of susceptible class is constant, the incidence and treatment rate are of  saturated type, susceptible class is vaccinated, the normal and disease induced death are also taken into consideration. It is also assumed that some of the infected individuals who are physically strong enough can recover themselves without treatment. Incorporating all the assumptions the governing differential equations of the model can be written in the following form}}
\begin{equation}\label{eq:1}      
\begin{cases}
\frac{dS}{dt} &= A-\frac{\beta SI}{1+\alpha I}-dS-u_{1}S\\
\frac{dI}{dt} &= \frac{\beta SI}{1+\alpha I}-(d+\delta+\gamma)I-\frac{ru_{2}I}{1+bu_{2}I}\\
\frac{dR}{dt} &=\frac{ru_{2}I}{1+bu_{2}I}+\gamma I+u_{1}S-dR
\end{cases}
\end{equation}
with initial conditions $S(0) \geq 0$, $I(0) \geq 0$, $R(0) \geq 0$. Parameters used in the system (1) are non-negative and listed in Table 1.
\begin{center}
	\begin{tabular}{| m{1.5cm} | m{8cm}|}
		\hline
		\begin{center}
			Parameters
		\end{center} & \begin{center}
			Interpretations
		\end{center} \\
		\hline
		$ A $ & Recruitment rate of the population.\\
		\hline
		$ \beta $ & Transmission rate. \\
		\hline
		$ \alpha $ & The parameter that measures the inhibitory factors. \\
		\hline
	    $ d $ & The natural mortality rate of the populations.   \\
	    \hline
	    $ \delta $ & Disease induced death rate.\\
		\hline
		$ \gamma $ & The natural recovery rate of the infected individuals.  \\
		\hline
		$ r $ & Cure rate.\\
		\hline
		$ b $ & Delayed parameter of treatment.\\
		\hline
		$ u_{1} $ & The control variable, be the percentage of susceptible individuals being vaccinated per unit of time.\\
		\hline
		$ u_{2} $ & The treatment control parameter.\\
		\hline
		\end{tabular}
\end{center}
\begin{center}
Table 1. Model parameters and their descriptions
\end{center}
\paragraph{\textnormal{Since the exact solution of the non linear autonomous system (1) is impossible to find, so we are analysing the qualitative behaviour of the solutions in the neighbourhood of the equilibrium points.}}
\section{Boundedness of Solutions, Existence of the Equilibria and the Basic Reproduction Number}
In this section we shall discuss the boundedness of the solutions and existence of the equilibrium points of system (1) for fixed value of control parameters $ u_{1} $ and $ u_{2} $. We shall derive the basic reproduction number when the control parameters are taken as fixed.
\begin{lemma}
The region $D = \big\{(S,I,R)\in \mathbb{R}_{+}^{3} /S+I+R \leq \frac{A}{d} \big\}$ is a positively invariant set for the model (1).
\end{lemma}   
\begin{proof}: Let $ N=S+I+R. $\\ So, $ \frac{dN}{dt}=A-dN-\delta I \leq A-dN $, integrating and taking $limsup$ as $t\to \infty $ we get  $\limsup\limits_{t\to \infty} N(t)\leq \frac{A}{d}$.  \\ Hence the lemma is proved.
\end{proof}    
\paragraph{\textnormal{The system (1) has always the disease free equilibrium point (DFE) $ A_{1}(S_{1},0,R_{1})=\bigg(\frac{A}{d+u_{1}},0,\frac{u_{1}A}{d(d+u_{1})}\bigg)$ at which the population remains in the absence of disease. Therefore, the model (1) has a threshold parameter $ R_{0} $, known as the basic reproduction number, which is defined as the number of secondary infection produced by a single infection in a completely susceptible population.   }}
\begin{lemma}
The basic reproduction number for the model (1) is $ R_{0}=\frac{\beta A}{(d+u_{1})(d+\delta+\gamma+ru_{2})}$.
\end{lemma}
\begin{proof}: Here is only one infected compartment, that is, the variable $ I $ and the disease free equilibrium point is $ A_{1} $. The basic reproduction number $ R_{0} $ is defined as the spectral radius of the next generation matrix $ FV^{-1} $ with small domain [13], where $ F=\left[ \Bigg(\frac{\beta S}{(1+\alpha I)^2}\Bigg)_{1\times1} \right] _{DFE} = \Bigg(\frac{\beta A}{d+u_{1}}\Bigg)_{1\times1}$ and $ V=\left[\Bigg(d+\delta+\gamma+\frac{ru_{2}}{(1+bu_{2} I)^2}\Bigg)_{1\times1}\right] _{DFE} = \Bigg(d+\delta+\gamma+ru_{2}\Bigg)_{1\times1}$. Thus, $ R_{0} $ of the model is $ \frac{\beta A}{(d+u_{1})(d+\delta+\gamma+ru_{2})}. $  \\Hence the lemma is proved.
\end{proof} 
The existence of the endemic equilibrium point(s) can be determined by the relation \\
$ S=\frac{ru_{2}(1+\alpha I)}{\beta (1+bu_{2}I)}+\frac{(d+\delta+\gamma)(1+\alpha I)}{\beta}=\frac{A(1+\alpha I)}{\beta I+(d+u_{1})(1+\alpha I)} $. Thus, the compartment $ I $ of the equilibrium point $ (S,I,R) $ must satisfy the equation
\begin{equation}\label{eq:2}   
H(I)=0, 
\end{equation}
where $ H(I)\equiv\frac{A}{\beta I+(d+u_{1})(1+\alpha I)}-\frac{ru_{2}}{\beta (1+bu_{2}I)}-\frac{d+\delta+\gamma}{\beta}. $\\
By simplifying the equation (2), we get 
\begin{equation}\label{eq:3}   
C_{1}I^2+C_{2}I+C_{3}=0, 
\end{equation}
where $ C_{1}=bu_{2}(d+\delta+\gamma)\big\{\beta +\alpha(d+u_{1})\big\}, $ \\ $ C_{2}= bu_{2}\big\{(d+u_{1})(d+\delta+\gamma)-\beta A\big\}+(d+\delta+\gamma+ru_{2})\big\{\beta +\alpha(d+u_{1})\big\},$\\ $C_{3}=(d+u_{1})(d+\delta+\gamma+ru_{2})(1-R_{0}).$\\ 
Here, the coefficient $ C_{1} $ is always positive and the sign of $ C_{3} $ depends only on the value of $ R_{0} .$ Thus, we have \\
\textbf{(1)} If $ R_{0}>1 $ , then only one endemic equilibrium point $ (S^{*},I^{*},R^{*}) $ exists.\\
\textbf{(2)} If $ C_{2}>0 $ and $ R_{0}<1 $ , then there is no endemic equilibrium point.\\
\textbf{(3)} If $ C_{2}<0,$ $ C_{2}^2-4C_{1}C_{3}>0 $ and $ R_{0}<1 $ , then two endemic equilibrium points $ (S_{1}^{*},I_{1}^{*},R_{1}^{*}) $ and $ (S_{2}^{*},I_{2}^{*},R_{2}^{*}) $ exist with $ I_{1}^{*}<I_{2}^{*} $.
\section{Stability and Bifurcation Analysis at DFE}
In this section we shall investigate the stability and transcritical bifurcation at the disease free equilibrium point(DFE) for fixed vaccination and treatment control. Here, the variational matrix corresponding to the system (\ref{eq:1}) is \\
\begin{equation}\label{eq:4}   
J(S,I,R)=\left( \begin{array}{ccc} -\frac{\beta I}{1+\alpha I}-d-u_{1} & -\frac{\beta S}{(1+\alpha I)^2} & 0 \\
 \frac{\beta I}{1+\alpha I} & \frac{\beta S}{(1+\alpha I)^2}-(d+\delta+\gamma)-\frac{ru_{2}}{(1+bu_{2} I)^2} & 0 \\ u_{1} &     \frac{ru_{2}}{(1+bu_{2} I)^2}+\gamma & -d \end{array} \right).
\end{equation}
\begin{theorem}
If $ R_{0}<1 $ then the disease free equilibrium point $ A_{1} $ is asymptotically stable and if $ R_{0}>1 $ then it is unstable.
\end{theorem}
\begin{proof}: The characteristic roots of the variational matrix (4) at the disease free equilibrium point $ A_{1} $ are $-d,-(d+u_{1})$ and $(d+\delta+\gamma+ru_{2})(R_{0}-1)$. Therefore, $ A_{1} $ is asymptotically stable when $ R_{0}<1 $ and is unstable when $ R_{0}>1 $. \\ Hence the theorem is proved.
\end{proof}
\begin{theorem}
For $ R_{0}=1 $, $ A_{1} $ is asymptotically stable if $(d+\delta+\gamma+ru_{2})\big\{\beta+(d+u_{1})\alpha\big\}<(d+u_{1})rbu_{2}^2$ and is unstable if $(d+\delta+\gamma+ru_{2})\big\{\beta+(d+u_{1})\alpha\big\}>(d+u_{1})rbu_{2}^2$.
\end{theorem}
\begin{proof}: For $ R_{0}=1 $, the eigen values of the variational matrix (4) at the equilibrium point $ A_{1}(S_{1},0,R_{1}) $ are $ 0,-(d+u_{1}),-d $. So, $ A_{1} $ is a non-hyperbolic equilibrium point and Centre manifold theory will be applied to determine its stability. \\
Putting $ S'=S-S_{1},I'=I,R'=R-R_{1} $ in the system (\ref{eq:1}) and using Taylor's expansion we get (omitting the `dash' sign) 
\begin{equation}\label{eq:5}      
\frac{dX}{dt} = BX+F(S,I,R),
\end{equation}
where $ B=\left( \begin{array}{ccc} -(d+u_{1}) & -\frac{\beta A}{d+u_{1}} & 0\\
0 & 0 & 0 \\ u_{1} & (ru_{2}+\gamma) & -d \end{array} \right),X=
\left( \begin{array}{c} S \\ I \\ R \end{array} \right), F(S,I,R)=\left( \begin{array}{c} -\beta SI+\alpha \beta S_{1}I^2 \\ \beta SI+(rbu_{2}^2-\alpha \beta S_{1})I^2 \\ -rbu_{2}^2I^2 \end{array} \right).$
\begin{center}
(Neglecting the terms of order $ \geq 3 $)
\end{center}
Now, we construct a matrix $ P=\left( \begin{array}{ccc} \beta Ad & 1 & 0\\
-d(d+u_{1})^2 & 0 & 0 \\ u_{1}\beta A-(ru_{2}+\gamma)(d+u_{1})^2 & -1 & 1 \end{array} \right)$ so that $ P^{-1}BP= diag \bigg(0,-(d+u_{1}),-d\bigg) $. By using the transformation $ X=PY $,where $ Y=
\left( \begin{array}{c} S' \\ I' \\ R' \end{array} \right)$, the system (5) can be transformed into the form (omitting the `dash' sign)\\
\begin{equation}\label{eq:6}      
\begin{cases}
\frac{dS}{dt} &= 0+g_{11}(S,I,R)\\
\frac{dI}{dt} &= -(d+u_{1})I+g_{22}(S,I,R)\\
\frac{dR}{dt} &= -dR+g_{33}(S,I,R)
\end{cases}
\end{equation}\\where $ g_{11}(S,I,R)=A_{11}S^2+B_{11}SI,g_{22}(S,I,R)=A_{22}S^2+B_{22}SI,g_{33}(S,I,R)=A_{33}S^2+B_{33}SI $, with $A_{11}=\big\{\beta ^2 A+(d+u_{1})\alpha \beta A-(d+u_{1})^2rbu_{2}^2\big\}d$.\\
By the Centre manifold theory [14], there exists a centre manifold of the system (\ref{eq:6}) which can be expressed by \\ $ W^{c}(0)=\big\{(S,I,R)/I=h_{1}(S),R=h_{2}(S) for S<\delta\big\} $, where $ \delta (>0) $ is some number and $ h_{1}(0)=h_{2}(0)=0,Dh_{1}(0)=Dh_{2}(0)=0 $. To compute the centre manifold $ W^{c}(0)$, we assume that $ I=h_{1}(S)=h_{11}S^2+h_{12}S^3+........... ,  R=h_{2}(S)=h_{21}S^2+h_{22}S^3+........... $ . So from the Local Centre manifold theorem we have the flow on the centre manifold $ W^{c}(0)$ defined by the differential equation \\
\begin{equation}\label{eq:7}      
\frac{dS}{dt} = A_{11}S^2+\frac{A_{11}B_{11}}{(d+u_{1})}S^3+\frac{A_{22}B_{11}(B_{22}-2A_{11})}{(d+u_{1})^2}S^4+...... .
\end{equation}\\
Now, by the condition $ R_{0}=1 $, we get after simplification that $A_{11}=d(d+u_{1})\Big[(d+\delta+\gamma+ru_{2})\big\{\beta+(d+u_{1})\alpha\big\}-(d+u_{1})rbu_{2}^2\Big]$. The system will be stable and unstable according as $ A_{11} < 0 $ and $ A_{11} > 0 $ respectively.
\\Hence the theorem is proved.
\end{proof}
\textbf{Note}: Other components $ B_{11}, A_{22}, B_{22}, A_{33} $ and $ B_{33} $ are not derived here as they are not in use.
\begin{theorem}
If $ R_{0}<1 $ and $ \alpha \geq bu_{2} $ then the disease free equilibrium point $ A_{1} $ is globally asymptotically stable.
\end{theorem}
\begin{proof}: We rewrite the system (1) in $ (S,I) $ plane as given below \\
\begin{equation}\label{eq:8}      
\begin{cases}
\frac{dS}{dt} &= A-\frac{\beta SI}{1+\alpha I}-dS-u_{1}S \equiv F(S,I)\\
\frac{dI}{dt} &= \frac{\beta SI}{1+\alpha I}-(d+\delta+\gamma)I-\frac{ru_{2}I}{1+bu_{2}I} \equiv G(S,I).
\end{cases}
\end{equation}\\
Now considering the Dulac function $B(S,I)=\frac{1+bu_{2}I}{SI}$, we get
\begin{equation*}
\frac{\partial (BF)}{\partial S}+\frac{\partial (BG)}{\partial I}=-\frac{A(1+bu_{2}I)}{IS^2}-\frac{(d+\delta+\gamma)bu_{2}}{S}-\frac{\beta (\alpha-bu_{2})}{(1+\alpha I)^2}.
\end{equation*}
Thus the DFE is globally asymptotically stable if $ \alpha \geq bu_{2} $ . Hence the theorem is proved.
\end{proof}
In other words DFE is globally asymptotically stable if the inhibitory coefficient exceeds a value that is the product of delayed parameter of treatment and the treatment control.
\begin{theorem}
If $ \beta A>(d+u_{1})(d+\delta+\gamma) $, then the system (1) experiences a transcritical bifurcation at $ A_{1} $ as $ u_{2} $ varies through the bifurcation value $ u_{2}^{0}=\frac{\beta A}{r(d+u_{1})}-\frac{d+\delta+\gamma}{r} $.
\end{theorem}
\begin{proof}
Let $ f(S,I,R;u_{2})=
\left( \begin{array}{c} A-\frac{\beta SI}{1+\alpha I}-dS-u_{1}S \\ \frac{\beta SI}{1+\alpha I}-(d+\delta+\gamma)I-\frac{ru_{2}I}{1+bu_{2} I} \\ \frac{ru_{2}I}{1+bu_{2} I}+\gamma I+u_{1}S-dR \end{array} \right) , u_{2}^{0}=\frac{\beta A}{r(d+u_{1})}-\frac{d+\delta+\gamma}{r} $. \\So, $ Df(A_{1},u_{2}^{0})=\left( \begin{array}{ccc} -(d+u_{1}) & -\frac{\beta A}{d+u_{1}} & 0\\
0 & 0 & 0 \\ u_{1} & \frac{\beta A}{d+u_{1}}-(d+\delta) & -d \end{array} \right) $. Clearly, $ f(A_{1},u_{2}^{0})=0 $ and $ Df(A_{1},u_{2}^{0})$ has a simple eigen value $ \lambda =0. $ Thus, we shall use Sotomayor theorem [14] to establish the existence of transcritical bifurcation. Now, a eigen vector of $ Df(A_{1},u_{2}^{0})$ corresponding to the eigen value  $ \lambda =0$ is $ V=\left( \begin{array}{c} 1 \\ -\frac{(d+u_{1})^2}{\beta A} \\ \frac{(d+\delta)(d+u_{1})^2}{\beta Ad}-1 \end{array} \right) $ and a eigen vector of $ \big(Df(A_{1},u_{2}^{0}\big))^{T}$ corresponding to the eigen value  $ \lambda =0$ is $ W=\left( \begin{array}{c} 0 \\ 1 \\ 0 \end{array} \right) $. Let $ f_{u_{2}} $ denote the vector of partial derivatives of the components of $ f $ with respect of $ u_{2} $. Thus $ f_{u_{2}}=\left( \begin{array}{c} 0 \\ -\frac{rI}{(1+bu_{2}I)^{2}} \\ \frac{rI}{(1+bu_{2}I)^{2}} \end{array} \right) $ and so $ f_{u_{2}}(A_{1},u_{2}^{0})=\left( \begin{array}{c} 0 \\ 0 \\ 0 \end{array} \right) $. 
\\ Therefore, $ W^{T}f_{u_{2}}(A_{1},u_{2}^{0})= \textbf{0} $,\\
$ W^{T}\big(Df_{u_{2}}(A_{1},u_{2}^{0})V\big)=\left( \begin{array}{c} \frac{r(d+u_{1})^{2}}{\beta A} \end{array} \right) \neq \textbf{0} $ and \\
$ W^{T}\big(D^{2}f(A_{1},u_{2}^{0})(V,V)\big)=2(d+u_{1})^{2}\left( \begin{array}{c} -\frac{1}{A}-\frac{\alpha (d+u_{1})}{\beta A}+\frac{b}{r}+\frac{b(d+\delta+\gamma)^{2}(d+u_{1})^{2}}{rA^{2}\beta ^2}-\frac{2b(d+\delta+\gamma)(d+u_{1})}{rA \beta} \end{array} \right) \neq \textbf{0} $.\\
Therefore, all the conditions for transcritical bifurcation in Sotomayor theorem are satisfied. Hence, the system (1) experiences a transcritical bifurcation at the equilibrium point $ A_{1} $ as the parameter $ u_{2} $ varies through the bifurcation value $ u_{2}=u_{2}^{0} $. Hence the theorem is proved.
\end{proof}
\section{Backward Bifurcation and Stability Analysis of Endemic Equilibria}
In this section, we shall analyse the stability and the bifurcation behaviour at endemic equilibrium point by assuming that two controls $ u_{1} $ and $ u_{2} $ are constant. We have already proved that DFE is stable if $ R_{0}<1 $ and is unstable if $ R_{0}>1 $. Here, we shall establish that the bifurcating endemic equilibrium exists for $ R_{0}<1 $, which implies that the backward bifurcation occurs. Now, we shall obtain the necessary and sufficient condition on model parameters for the existence of backward bifurcation.
\begin{theorem}
The system (1) has a backward bifurcation at $ R_{0}=1 $ if and only if $ (ru_{2}+d+\delta+\gamma)(ru_{2}+d+\delta+\gamma+\alpha A)<bru_{2}^{2}A $.
\end{theorem}
\begin{proof}: In Section 3, we have seen that the infected component $ I $ of endemic equilibrium points are the roots of the equation (3). Again, from Lemma 2 we can express $ \beta $ as $ \frac{R_{0}(d+u_{1})(d+\delta+\gamma+ru_{2})}{A}. $ Now, we substitute $ \beta $ in the coefficients of equation (3) and rewrite equation (3) as 
\begin{equation*}
C_{1}I^2+C_{2}I+C_{3}=0 ,
\end{equation*}
where $ C_{1}=\frac{bu_{2}(d+\delta+\gamma)(d+u_{1})\lbrace R_{0} (d+\delta+\gamma+ru_{2})+\alpha A\rbrace}{A}, C_{2}= bu_{2}(d+u_{1})\big\{(d+\delta+\gamma)-R_{0}(d+\delta+\gamma+ru_{2})\big\}+\frac{(d+\delta+\gamma+ru_{2})(d+u_{1})\lbrace R_{0} (d+\delta+\gamma+ru_{2})+\alpha A\rbrace}{A} $ and  $C_{3}=(d+u_{1})(d+\delta+\gamma+ru_{2})(1-R_{0}).$ \\
To obtain a necessary and sufficient condition on the model parameters such that backward bifurcation occurs we have to compute the value of $\left[ \frac{\partial I}{\partial R_{0}} \right]_{R_{0}=1,I=0} $. Now, differentiating the equation (3) implicitly with respect to $R_0$ we obtain
\begin{equation*}
\left[\frac{\partial I}{\partial R_{0}}\right] _{R_{0}=1,I=0}= \frac{A(d+\delta+\gamma+ru_{2})}{(d+\delta+\gamma+ru_{2})(d+\delta+\gamma+ru_{2}+\alpha A)-bru_{2}^{2}A} .
\end{equation*}
The system (1) has a backward bifurcation at $ R_{0}=1 $ if and only if the value of the slope $\left[ \frac{\partial I}{\partial R_{0}} \right]_{R_{0}=1,I=0} $ of the curve $ I=I(R_{0}) $ is less than zero. Hence we obtain the necessary and sufficient condition for backward bifurcation in the form $ (ru_{2}+d+\delta+\gamma)(ru_{2}+d+\delta+\gamma+\alpha A)<bru_{2}^{2}A .$\\Hence the theorem is proved.
\end{proof}
\begin{figure}[H]
 \centering
  {\includegraphics[width=3.5 in]{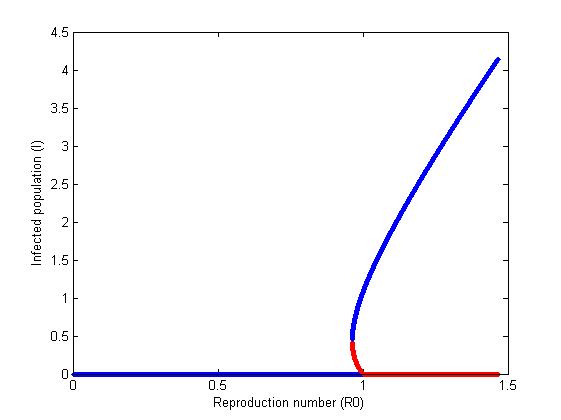}}
 \caption{Backward bifurcation curve for the parametric values $ A=11,\alpha=0.5,d=0.000039,\gamma=0.08,\delta=0.02,r=0.4,b=2.21,u_{1}=0.5,u_{2}=0.5 $.} 
\end{figure}
So, there is a real number $ R_{0}^{*}<1 $ for which two endemic equilibria exist for  $  R_{0}^{*}<R_{0}<1 $ if the condition in Theorem 5 holds. Now, we shall focus on the stability analysis of the endemic equilibrium point(s) for different values of $ R_{0} $ and we shall prove in the following theorem that the locally asymptotically stable DFE co-exists with a locally asymptotically stable endemic equilibrium point when $ R_{0} < 1 $ .
\begin{theorem}
If $ R_{0}>1 $ and $ \beta \geq max\lbrace rbu_{2}^{2}, r \alpha u_{2} \rbrace $, then the system (1) has a unique endemic equilibrium point $ (S^{*},I^{*},R^{*}) $ that is locally asymptotically stable. On the other hand if $ R_{0}^{*}<R_{0}<1 $ and $ (ru_{2}+d+\delta+\gamma)(ru_{2}+d+\delta+\gamma+\alpha A)<bru_{2}^{2}A $ , then the system (1) has two endemic equilibrium points. The one with the smaller number of infecteds, $ (S_{1}^{*},I_{1}^{*},R_{1}^{*}) $, is unstable, while the other , with a higher number of infecteds, $ (S_{2}^{*},I_{2}^{*},R_{2}^{*}) $, is locally asymptotically stable if $ \beta \geq max\lbrace rbu_{2}^{2}, r \alpha u_{2} \rbrace $.
\end{theorem}
\begin{proof}
We have the following characteristic equation of the variational matrix (4) at the endemic equilibrium point $ (S,I,R) $.
\[
\begin{vmatrix} -\frac{\beta I}{1+\alpha I}-d-u_{1}-\lambda & -\frac{\beta S}{(1+\alpha I)^2} & 0 \\
 \frac{\beta I}{1+\alpha I} & \frac{\beta S}{(1+\alpha I)^2}-(d+\delta+\gamma)-\frac{ru_{2}}{(1+bu_{2} I)^2}-\lambda & 0 \\ u_{1} &     \frac{ru_{2}}{(1+bu_{2} I)^2}+\gamma & -d-\lambda \end{vmatrix}=0,
\] \\ or equivalently,
\[
(\lambda+d)\begin{vmatrix} -\frac{\beta I}{1+\alpha I}-d-u_{1}-\lambda & -\frac{\beta S}{(1+\alpha I)^2} \\
 \frac{\beta I}{1+\alpha I} & \frac{\beta S}{(1+\alpha I)^2}-(d+\delta+\gamma)-\frac{ru_{2}}{(1+bu_{2} I)^2}-\lambda \end{vmatrix}=0.
\] \\
So, one of the three eigen values of the variational matrix is $ -d. $ The remaining eigen values are the solutions of the equation $ G(\lambda)=0, $ \\where
\[
G(\lambda) \equiv \begin{vmatrix} -\frac{\beta I}{1+\alpha I}-d-u_{1}-\lambda & -\frac{\beta S}{(1+\alpha I)^2} \\
 \frac{\beta I}{1+\alpha I} & \frac{\beta S}{(1+\alpha I)^2}-(d+\delta+\gamma)-\frac{ru_{2}}{(1+bu_{2} I)^2}-\lambda \end{vmatrix}.
\]
So, 
\[
G(0)=\begin{vmatrix} -\frac{\beta I}{1+\alpha I}-d-u_{1} & -\frac{\beta S}{(1+\alpha I)^2} \\
 \frac{\beta I}{1+\alpha I} & \frac{\beta S}{(1+\alpha I)^2}-(d+\delta+\gamma)-\frac{ru_{2}}{(1+bu_{2} I)^2} \end{vmatrix}.
\]
Again, we know from (2) that the component $ I $ of the endemic equilibrium point(s) are the solutions of the equation $ H(I)=0 $ and $ H(0)=(\frac{d+\delta+\gamma+ru_{2}}{\beta})(R_{0}-1) $. Thus, $ R_{0}>1 $ if and only if $ H(0)>0 $ and $ R_{0}<1 $ if and only if $ H(0)<0 $. Now, we shall derive the relation between $ G(0) $ and $ H'(I). $ Differentiating $ H(I) $ with respect to $ I $, we get
$ H'(I) \equiv \frac{rbu_{2}^{2}}{\beta (1+bu_{2}I)^{2}}-\frac{A\lbrace\beta+\alpha (d+u_{1})\rbrace}{\lbrace\beta I+(1+\alpha I)(d+u_{1})\rbrace^{2}}. $\\
Now, 
\begin{center}
$ G(0)$=$\begin{vmatrix} -\frac{\beta I}{1+\alpha I}-d-u_{1} & -\frac{\beta S}{(1+\alpha I)^2} \\
 \frac{\beta I}{1+\alpha I} & \frac{rbu_{2}^{2}I}{(1+bu_{2} I)^2}-\frac{\alpha \beta SI}{(1+\alpha I)^2} \end{vmatrix} $ \\ \
= $\begin{vmatrix} -\frac{\beta I}{1+\alpha I}-d-u_{1} & -\frac{\beta S}{(1+\alpha I)^2} \\
-d-u_{1} & \frac{rbu_{2}^{2}I}{(1+bu_{2} I)^2}-\frac{\beta S}{1+\alpha I} \end{vmatrix} $ \\ \
= $(-\frac{A}{S})\begin{vmatrix} 1 & \frac{\beta A}{\lbrace\beta I+(d+u_{1})(1+\alpha I)\rbrace^{2}}\\
 -(d+u_{1}) & \frac{rbu_{2}^{2}I}{(1+bu_{2} I)^2}-\frac{\beta S}{1+\alpha I} \end{vmatrix} $ \\ \
=$(-\frac{A}{S})\begin{vmatrix} 1 & \frac{\beta A}{\lbrace\beta I+(d+u_{1})(1+\alpha I)\rbrace^{2}}\\
 -(d+u_{1}) & \frac{rbu_{2}^{2}I}{(1+bu_{2} I)^2}-\frac{\beta A}{\beta I+(d+u_{1})(1+\alpha I)} \end{vmatrix}$  =$(-\frac{\beta AI}{S})H'(I)$. 
\end{center}
So, we have $ G(0)>0 $ if and only if $ H'(I)<0 $ and 
$ G(0)<0 $ if and only if $ H'(I)>0. $\\
Now, we shall discuss two cases.\\
\textbf{Case I}: Suppose $ R_{0}>1 $. Then $ H(0)>0 $. We have already proved in Section 3 that when $ R_{0}>1 $ then only one endemic equilibrium point $ (S^{*},I^{*},R^{*}) $ exists. Since $ H(0)>0 $, hence $ H(I) $ should decrease in some neighbourhood of $ I^{*} $. Thus, in this case $ H'(I^{*})<0 $ and so $ G(0)>0 $. Again, we know that one of the eigen values of the variational matrix is $ -d $ and the remaining eigen values are the solutions of the equation $ G(\lambda)=0 $ i.e. the equation $ \lambda^{2}+K_{1} \lambda+K_{2}=0, $ where $ K_{1}=2d+\delta+\gamma+u_{1}+\frac{ru_{2}}{(1+bu_{2}I)^{2}}+\frac{\beta I}{1+\alpha I}-\frac{\beta S}{(1+\alpha I)^{2}} $ and $ K_{2}=G(0). $ By using the relation $ S=\frac{ru_{2}(1+\alpha I)}{\beta (1+bu_{2}I)}+\frac{(d+\delta+\gamma)(1+\alpha I)}{\beta} $, $ K_{1} $ is simplified as
\[
K_{1}=d+u_{1}+\frac{\alpha \beta SI}{(1+\alpha I)^{2}}+\frac{\beta I}{1+\alpha I}-\frac{rbu_{2}^{2} I}{(1+bu_{2} I)^{2}} \geq d+u_{1}+\frac{\alpha \beta SI}{(1+\alpha I)^{2}}+\frac{\beta I}{1+\alpha I}-\frac{rbu_{2}^{2} I}{1+bu_{2} I} \]
\[= d+u_{1}+\frac{\alpha \beta SI}{(1+\alpha I)^{2}}+\frac{I}{(1+\alpha I)(1+bu_{2}I)} \lbrace (\beta -rbu_{2}^{2})+(\beta bu_{2}-\alpha rbu_{2}^{2})I \rbrace .
\]
Thus, $ K_{1} $ is positive if the condition $ \beta \geq max\lbrace rbu_{2}^{2}, r \alpha u_{2} \rbrace $ holds and $ K_{2}=G(0)> 0 $. Hence, all the eigen values of the variational matrix have negative real part. Therefore, $ (S^{*},I^{*},R^{*}) $ is asymptotically stable if $ \beta \geq max\lbrace rbu_{2}^{2}, r \alpha u_{2} \rbrace $. \\
\textbf{Case II}: Suppose $ R_{0}^{*}<R_{0}<1 $. Then $ H(0)<0.$ Again, we have already proved that two endemic equilibria $ (S_{1}^{*},I_{1}^{*},R_{1}^{*}) $ and $ (S_{2}^{*},I_{2}^{*},R_{2}^{*}) $ (with $ I_{1}^{*}<I_{2}^{*} $)  exist for  $  R_{0}^{*}<R_{0}<1 $ if  the condition $ (ru_{2}+d+\delta+\gamma)(ru_{2}+d+\delta+\gamma+\alpha A)<bru_{2}^{2}A $ holds. So, the function $ H(I) $ must increase in some neighbourhood of $ I_{1}^{*} $ and decrease in some neighbourhood of $ I_{2}^{*} $. Therefore, $ H'(I_{1}^{*})>0 $  and $ H'(I_{2}^{*})<0 .$ In this case, we have reached following two conclusions. \\
(1) For the equilibrium point $ (S_{1}^{*},I_{1}^{*},R_{1}^{*}) $, we have $ H'(I_{1}^{*})>0 $. So, $ G(0)<0. $ Again, $\lim G(\lambda)=\infty $ as $ \lambda \rightarrow \infty.$ Thus, $ G(\lambda_{i})=0 $ for some $ \lambda_{i}>0 .$ So, at least one eigen value of the variational matrix is positive. Therefore, $ (S_{1}^{*},I_{1}^{*},R_{1}^{*}) $ is unstable.\\
(2) For the equilibrium point $ (S_{2}^{*},I_{2}^{*},R_{2}^{*}) $, we have $ H'(I_{2}^{*})<0 $ and so $ G(0)>0. $ Thus, we proceed same as case I and derive that $ (S_{2}^{*},I_{2}^{*},R_{2}^{*}) $ is asymptotically stable if $ \beta \geq max\lbrace rbu_{2}^{2}, r \alpha u_{2} \rbrace $.
\\Hence the theorem is proved.
\end{proof}
In Figure 1, we have plotted backward bifurcation curve where blue and red lines represent the lines of stable and unstable equilibrium points respectively. Therefore, Theorem 6 is justified by Figure 1.
\section{Characterization of the Optimal Control}
In this model, we have considered two controls, one control variable $u_{1}$ is used for vaccinating the susceptible populations and other control variable $u_{2}$ is used for treatment efforts for infected individuals. We assume that both vaccination and treatment controls are the functions of time $ t $ as they are applied according to the necessity. Our main objective is to minimize the total loss occurs due to the presence of infection and the cost due to vaccination of susceptible individuals and treatment of infected individuals. Thus, the strategy of the optimal control is to minimize the susceptible and infected individuals as well as the cost of implementing the two controls. Thus, we construct the objective functional to be minimized as follows :
\begin{center}
$ J(u_{1},u_{2})=\int_{0}^{T}(A_{1}S+A_{2}I+B_{1}u_{1}^{2}+B_{2}u_{2}^{2})dt $	
\end{center}
where the constants $ A_{1} $ and $ A_{2} $ are respectively the per capita loss due to presence of susceptible and infected population at any time instant. Also, the constants $ B_{1} $ and $ B_{2} $ respectively represent the costs associated with vaccination of susceptible and treatment of infected individuals. We also assume that the time interval is $ [0,T] $. The problem is to find optimal functions $ (u_{1}^{*}(t),u_{2}^{*}(t)) $ such that 
$ J(u_{1}^{*},u_{2}^{*})=min \lbrace J(u_{1},u_{2}),(u_{1},u_{2})\in U\rbrace, $ where the control set is defined as $ U=\lbrace(u_{1},u_{2})/u_{i}(t)$ is Lebesgue measurable on $ [0,1],0\leq u_{1}(t),u_{2}(t)\leq 1,t\in [0,T]\rbrace $.
\begin{theorem}
There are optimal controls $ u_{1}^{*} $ and $ u_{2}^{*} $ such that $ J(u_{1}^{*},u_{2}^{*})=min\big\{ J(u_{1},u_{2}),(u_{1},u_{2})\in U\big\}$.
\end{theorem}
\begin{proof}: The integrand of the objective functional $ J(u_{1},u_{2}) $ is a convex function of $ u_{1} $ and $ u_{2} $. Since the solution of the system (1) is bounded, hence the system satisfies the Lipshitz property with respect to the variables $ S,I $ and $ R $. Therefore, there exists an optimal pair $ (u_{1}^{*},u_{2}^{*}) $ . \\
Hence the theorem is proved.
\end{proof}
The Lagrangian of the problem is given by $ L=A_{1}S+A_{2}I+B_{1}u_{1}^{2}+B_{2}u_{2}^{2} $. Now, we form the Hamiltonian $ H $ for the problem given by, 
\begin{center}
$ H(S,I,R,u_{1},u_{2},\lambda_{1},\lambda_{2},\lambda_{3})=A_{1}S+A_{2}I+B_{1}u_{1}^{2}+B_{2}u_{2}^{2}+\lambda_{1}(t)\lbrace A-\frac{\beta SI}{1+\alpha I}-dS-u_{1}S\rbrace+\lambda_{2}(t)\lbrace \frac{\beta SI}{1+\alpha I}-(d+\delta+\gamma)I-\frac{ru_{2}I}{1+bu_{2}I}\rbrace+\lambda_{3}(t)\lbrace \frac{ru_{2}I}{1+bu_{2}I}+\gamma I+u_{1}S-dR\rbrace $.
\end{center}
In order to determine the adjoint equations and transversality conditions, we use Pontryagin's Maximum Principle [15-16] which gives 
$ \frac{d\lambda_{1}(t)}{dt}=-\frac{\partial H}{\partial S} ,\frac{d\lambda_{2}(t)}{dt}=-\frac{\partial H}{\partial I}, \frac{d\lambda_{3}(t)}{dt}=-\frac{\partial H}{\partial R}, $ with the transversality conditions $ \lambda_{i}(T)=0,i=1,2,3. $ Thus, we have 
\begin{equation}\label{eq:9}      
\begin{cases}
\frac{d\lambda_{1}}{dt} &= -A_{1}+\frac{(\lambda_{1}-\lambda_{2})\beta I}{1+\alpha I}+d\lambda_{1}+u_{1}(\lambda_{1}-\lambda_{3})\\
\frac{d\lambda_{2}}{dt} &= -A_{2}+\frac{(\lambda_{1}-\lambda_{2})\beta S}{(1+\alpha I)^{2}}+\frac{(\lambda_{2}-\lambda_{3})ru_{2}}{(1+bu_{2}I)^{2}}+(d+\delta)\lambda_{2}+\gamma(\lambda_{2}-\lambda_{3})\\
\frac{d\lambda_{3}}{dt} &=d\lambda_{3}
\end{cases}
\end{equation}
with the transversality conditions 
\begin{equation}\label{eq:10}      
\lambda_{1}(T)=0, \lambda_{2}(T)=0, \lambda_{3}(T)=0.
\end{equation}
Now, using the optimality conditions $ \frac{\partial H}{\partial u_{1}}=0 $ and $ \frac{\partial H}{\partial u_{2}}=0 $ we get \\
$ u_{1}=\frac{(\lambda_{1}-\lambda_{3})S}{2B_{1}}$  and $ u_{2}(1+bu_{2}I)^2=\frac{(\lambda_{2}-\lambda_{3})rI}{2B_{2}} .$ Clearly, $ \frac{\partial ^{2} H}{\partial u_{1}^{2}}>0 , \frac{\partial ^{2} H}{\partial u_{2}^{2}}>0 $ and $ \frac{\partial ^{2} H}{\partial u_{1}^{2}} \frac{\partial ^{2} H}{\partial u_{2}^{2}}-(\frac{\partial ^{2} H}{\partial u_{1} \partial u_{2}})^{2}>0. $\\
Therefore, the optimal problem is minimum at controls $ u_{1}^{*} $ and $ u_{2}^{*} $ where 
$ u_{1}^{*} =max \big\{0,min\big\{ \frac{(\lambda_{1}^{*}-\lambda_{3}^{*})S^{*}}{2B_{1}},1\big\} \big\}$ and $ u_{2}^{*} =max \big\{0,min\big\{ \overline{u_{2}},1\big\} \big\}$, where $\overline{u_{2}}$ is the non-negative root of the equation $ u_{2}(1+bu_{2}I^{*})^2=\frac{(\lambda_{2}^{*}-\lambda_{3}^{*})rI^{*}}{2B_{2}} .$ Here, $ S^{*},I^{*},R^{*} $ are respectively the optimum values of $ S,I,R $ and $ (\lambda_{1}^{*},\lambda_{2}^{*},\lambda_{3}^{*}) $ is the solution of the system (9) with the condition (10). Thus, we summarize the details in the following:
\begin{theorem}
The optimal controls $ u_{1}^{*} $ and $ u_{2}^{*} $ which minimize $ J $ over the region $ U $ are given by \\
$ u_{1}^{*} =max \big\{0,min\big\{ \frac{(\lambda_{1}^{*}-\lambda_{3}^{*})S^{*}}{2B_{1}},1\big\} \big\}$ and $ u_{2}^{*} =max \big\{0,min\big\{ \overline{u_{2}},1\big\} \big\}$, where $\overline{u_{2}}$ is the non-negative root of the equation $ u_{2}(1+bu_{2}I^{*})^2=\frac{(\lambda_{2}^{*}-\lambda_{3}^{*})rI^{*}}{2B_{2}}$.
\end{theorem}
\section{Numerical Simulations and Efficiency Analysis}
To justify the impact of optimal control, we have used the forward-backward sweep method to solve the optimality system numerically. This method combines the forward application of a fourth order Runge-Kutta method for the state system (1) with the backward application of a fourth order Runge-Kutta method for the adjoint system (9) and the transversality conditions (10). Here, we fixed up our problem for 20 months and assume that the vaccination and treatment are stopped after 20 months. The simulation which we carried out by using the parametric values given in Table 2 with the initial conditions $ S(0)=50, I(0)=4 $ and $ R(0)=0.01 $. \\
\begin{center}
	\begin{tabular}{| m{2cm} | m{0.8cm}| m{0.6cm}|  m{0.8cm}|  m{0.8cm}|  m{0.8cm}|  m{0.8cm}|  m{0.8cm}|  m{0.8cm}| m{0.8cm}| m{0.8cm}| m{0.8cm}| m{0.6cm}|}
		\hline
		
		\begin{center}
			Parameters
		\end{center} &
		$ A $ & $ \beta $ & $ \alpha $ &$ d $& $ \delta $&$ \gamma $&$ r $ &$ b $ &$ A_{1} $ &$ A_{2} $ &$ B_{1} $ &$ B_{2} $ \\
		\hline
		\begin{center}
			Values
		\end{center} 
		 & 100 & 0.1 & 0.5 & 0.004 & 0.02 & 0.7 & 0.4	& 0.05 & 0.01 & 0.08 & 0.8 & 0.1\\
		
		\hline
		\end{tabular}
\end{center}
\begin{center}
Table 2
\end{center}
Figure 2(a)-(c) show the time series of the susceptible $ (S) $, infected $ (I) $ and recovered $ (R) $ individuals both with and without control. Figure 3(a)-(b) represent the optimal control $ u_{1}^{*} $ and $ u_{2}^{*} $ respectively for the time interval $ [0 , 20] $. From Figure 2(a)-(c), we see that optimal controls due to vaccination and treatment are very effective for reducing the number of susceptible and infected individuals and so enhancing the number of recovered individuals significantly.
\begin{figure}[H]
 \centering
  \subfloat[]{\includegraphics[width=2.5 in]{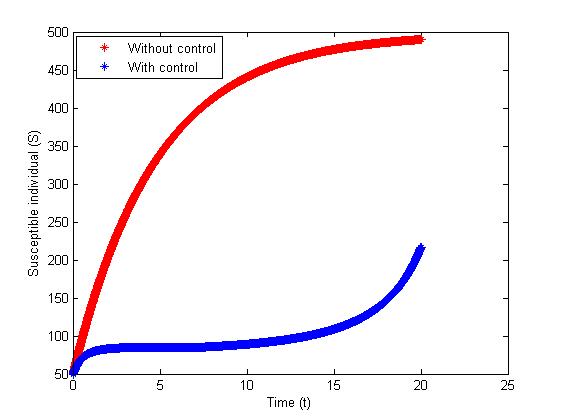}}
  \subfloat[]{\includegraphics[width=2.5 in]{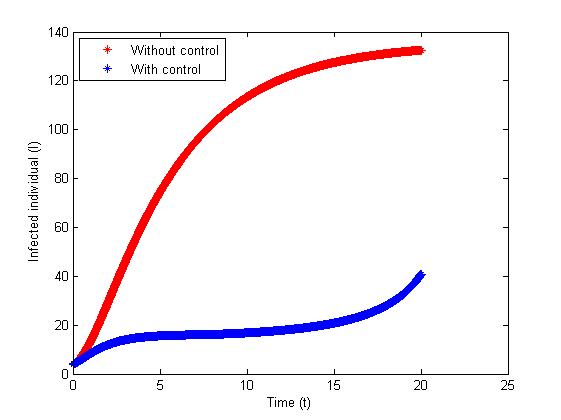}}
  \subfloat[]{\includegraphics[width=2.5 in]{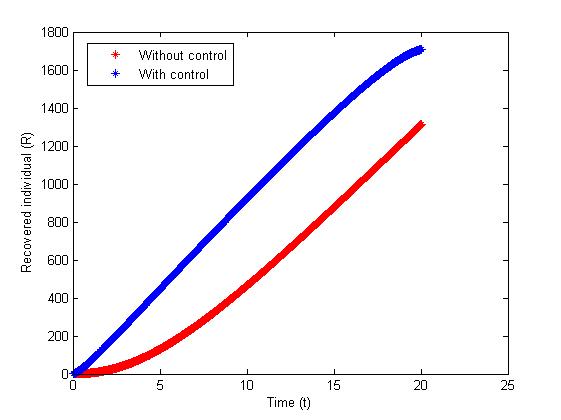}}
 \caption{Time series of the populations with control and without control:(a) susceptible individuals (b) infected individuals ,(c) recovered individuals .} 
\end{figure}
\begin{figure}[H]
 \centering
  \subfloat[]{\includegraphics[width=2.5 in]{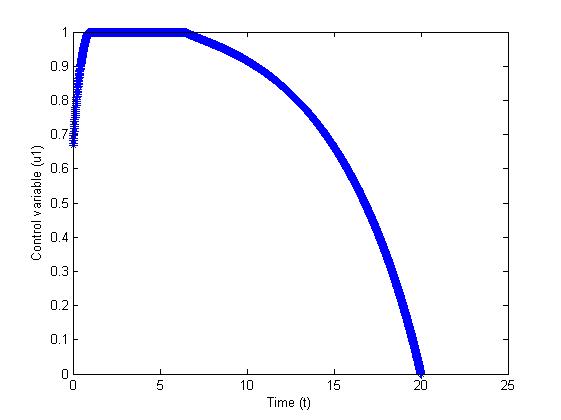}}
  \subfloat[]{\includegraphics[width=2.5 in]{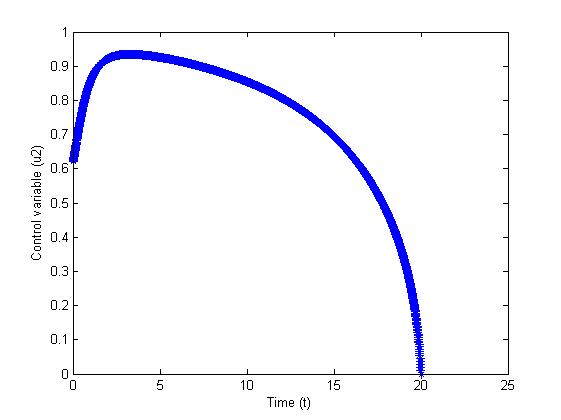}}
 \caption{Time series of control variables: (a) Optimal control $ u_{1} $, (b) Optimal control $ u_{2} $.}
\end{figure}
In this paper, we have considered two controls, one is vaccination control $ u_{1} $ and other is treatment control $ u_{2} $. But, if we use only one control among $ u_{1} $ and $ u_{2} $ then one question may arise `which control is more efficient to reduce infection ?'  To answer this question we will perform an efficiency analysis [17] which will allow us to determine the best control strategy. Here, we distinguish two control strategies STR-1 and STR-2 where STR-1 is the strategy where $ u_{1} \neq 0 $ , $ u_{2} = 0 $ and STR-2 is the strategy where $ u_{1} = 0 $ , $ u_{2} \neq 0 $. To determine the best control strategy among these two, we have to calculate the efficiency index (E.I.) = $ (1-\frac{\mathbb{A}^{c}}{\mathbb{A}^{o}})\times 100 $, where $ \mathbb{A}^{c} $ and $ \mathbb{A}^{o} $ are the cumulated number of infected individuals with and without control, respectively. The best strategy will be the one whom efficiency index will be bigger [17]. It can be noted that the cumulated number of infected individuals during the time interval $ [0 , 20] $ is defined by $ \mathbb{A} = \int_{0}^{20} I(t)dt $. We have used Simpson's $ \frac{1}{3} $ rule to evaluate the value of integration and we have $ \mathbb{A}^{0} = 1933.9 $. The values of $ \mathbb{A}^{c} $ and efficiency index (E.I.) for STR-1 and STR-2 are given in Table 3.
\begin{center}
	\begin{tabular}{| m{3cm} | m{3cm}| m{3cm} |}
		\hline
		\begin{center}
			Strategy
		\end{center} & \begin{center}
			$ \mathbb{A}^{c} $
		\end{center} & \begin{center}
			$ E.I. $
		\end{center}\\
		\hline
		\begin{center}
		$ STR-1 $ 
		\end{center}& \begin{center} $ 410.2195 $ 
		\end{center} & \begin{center} $ 78.79 $ \end{center}\\
		\hline
		\begin{center} $ STR-2 $ 
		\end{center}& \begin{center} $ 1787.7 $ \end{center} & \begin{center} $ 7.56 $ 
		\end{center}\\
		\hline
		\end{tabular}
\end{center}
\begin{center}
Table 3. Strategies and their efficiency index
\end{center}
From Table 3, it follows that STR-1 is the best strategy among STR-1 and STR-2 which permits to reduce the number of incident cases. Thus, vaccination is more effective than treatment.
\section{Conclusions}
In this paper, we have analysed the qualitative behaviour and optimal control strategy of an SIR model. We have introduced a saturated incidence rate which is affected by inhibitory factors and  considered a saturated treatment function which characterizes the effect of limited treatment capacity on the spread of infection. Two control functions have been used, one for vaccinating the susceptible populations and other for controlling the treatment efforts to the infected populations. To describe the complex dynamics of the solutions for constant controls, we have obtained the basic reproduction number $ R_{0} $ which plays a crucial role for the study of stability analysis of both disease free equilibrium point and endemic equilibrium points as well as backward bifurcation analysis. We have established that DFE is locally asymptotically stable for $ R_{0} < 1 $ and in addition, if inhibitory coefficient is greater than some quantity $ ( \alpha \geq bu_{2} ) $ then DFE is globally asymptotically stable which is very significant at the biological point of view. If $ R_{0} = 1 $, DFE is a non-hyperbolic equilibrium point and the stability analysis of this point has been investigated by using Centre manifold theory. We have also used Sotomayor theorem to show transcritical bifurcation at DFE with respect to the treatment control. We have obtained a necessary and sufficient condition on the model parameters such that backward bifurcation occurs. Moreover, stability analysis of endemic equilibrium points is discussed analytically for the different values of $ R_{0} $.\\
We have also studied and determined the optimal vaccination and treatment to minimize the number of infective and susceptible populations as well as the cost due to vaccination and treatment. A comparative study between the system with controls and without control has been presented to realize the positive impact of vaccination and treatment in controlling the infectious diseases. Finally, efficiency analysis has been performed to determine that the vaccinating to the susceptible populations is better than treatment control to infected populations in order to minimize the infected individuals. The entire study of this paper is mainly based on the deterministic framework and our proposed model is valid for large population only. The work is a theoretical modelling and it can be further justified using experimental results.

\end{document}